\documentclass[11pt]{amsart}

\usepackage{amssymb,amsmath,amscd,amsthm,enumerate,hyperref,verbatim}
\usepackage[utf8]{inputenc}
\usepackage[mathscr]{euscript}
\usepackage[left=1in,right=1in,top=1in,bottom=1in]{geometry}

\DeclareMathOperator{\capa}{cap}

\newcommand{\bbN}{{\mathbb{N}}}
\newcommand{\bbR}{{\mathbb{R}}}

\newcommand{\bbZ}{{\mathbb{Z}}}
\newcommand{\bbC}{{\mathbb{C}}}
\newcommand{\bbQ}{{\mathbb{Q}}}

\newcommand{\cB}{{\mathcal{B}}}

\newcommand{\cE}{{\mathcal{E}}}

\newcommand{\cJ}{{\mathcal{J}}}
\newcommand{\cK}{{\mathcal{K}}}

\newcommand{\cM}{{\mathcal{M}}}

\newcommand{\cS}{{\mathcal{S}}}
\newcommand{\cT}{{\mathcal{T}}}

\newcommand{\e}{{\varepsilon}}

\renewcommand{\Re}{\operatorname{Re}}

\theoremstyle{plain}
\newtheorem{thm}{Theorem}[section]
\newtheorem*{thm*}{Theorem}
\newtheorem{lem}[thm]{Lemma}
\newtheorem{prop}[thm]{Proposition}

\theoremstyle{definition}

\newtheorem*{rem}{Remark}
\newtheorem*{qstn}{Question}
\newtheorem*{ack}{Acknowledgments}

\numberwithin{equation}{section}
\numberwithin{thm}{section}
\numberwithin{defn}{section}

\begin{document}

\author[T.\ VandenBoom]{Tom VandenBoom}
\address{Department of Mathematics, Rice University, Houston, TX~77005, USA}
\email{tv4@rice.edu}

\thanks{T.V.\ was supported in part by NSF grants DMS--1361625 and DMS--1148609.}

\title{Reflectionless Discrete Schr\"odinger Operators Are Spectrally Atypical}

\begin{abstract}
We prove that, if an isospectral torus contains a discrete Schr\"odinger operator with nonconstant potential, the shift dynamics on that torus cannot be minimal.  Consequently, we specify a generic sense in which finite unions of nondegenerate closed intervals having capacity one are not the spectrum of any reflectionless discrete Schr\"odinger operator.  We also show that the only reflectionless discrete Schr\"odinger operators having zero, one, or two spectral gaps are periodic.
\end{abstract}

\maketitle

\section{Introduction and Main Results}

A discrete Schr\"odinger operator (DSO) is a self-adjoint linear operator $H$ on the Hilbert space of square-summable sequences $\ell^2(\bbZ)$ which acts entrywise via
\begin{align}
\label{eq:dso}
(Hu)_n = u_{n+1} + u_{n-1} + V(n)u_n, \;\; u \in \ell^2(\bbZ),
\end{align}
where $V$ is a bounded real potential funtion $V : \bbZ \to \bbR$.  A DSO $H_V$ is called almost-periodic if its potential function $V : \bbZ \to \bbR$ is almost-periodic; that is, if every sequence of translates  $V_k = V(\cdot + n_k)$ has a uniformly convergent subsequence.  A prominent example of an almost-periodic DSO is the Almost-Mathieu Operator given by $H_V$ with $V_{\lambda,\theta,\alpha}(n) = 2\lambda\cos(2\pi(\theta + n\alpha))$ for $\alpha \in [0,1]\setminus \bbQ$.  The DSO is most naturally identified as a discretized model of the differential Schr\"odinger operator, $L = -\Delta + V$.  The Schr\"odinger operator and DSO tend to share many of the same spectral characteristics, and as such the question of identifying the spectral characteristics of a DSO with a fixed almost-periodic potential $V$ is thoroughly studied and reasonably well-understood.

On the other hand, one can likewise ask the dual question: if one fixes certain spectral characteristics, what, if any, bounded self-adjoint operators demonstrate them? In this context, the Jacobi operator $J = J(a,b)$
\begin{align}
\label{eq:jacobi}
(Ju)_n = a_{n}u_{n-1} + b_nu_n + a_{n+1}u_{n+1}, \;\; u \in \ell^2(\bbZ)
\end{align}
arises in a very natural way.  Namely, when one fixes a compactly supported probability measure $\mu$ on $\bbR$, multiplication by the independent parameter $x$ with respect to the basis of orthogonal polynomials in $L^2(d\mu)$ is unitarily equivalent to operation by a half-line Jacobi operator $J$ on $\ell^2(\bbN)$.  Under the assumption of absolute continuity and almost-periodicity, a whole-line Jacobi operator can be completely reconstructed from any of its half-line restrictions.  In fact, for a finite union of nondegenerate disjoint closed intervals $E \subset \bbR$ (henceforth, a ``finite-gap compact"), the set $\cJ(E)$ of all absolutely continuous, almost-periodic, whole-line Jacobi operators with spectrum $E$ is naturally homeomorphic to a finite-dimensional torus \cite{GHMT08, SY97}.  In this case we call $\cJ(E)$ the isospectral torus for $E$. 

For finite-gap compacts $E$, certain potential-theoretic properties are directly related to spectral properties of elements of $\cJ(E)$.  If $\cM_1(E)$ denotes the set of Borel probability measures supported on $E$, define the logarithmic capacity of $E$ by
\begin{align*}
\capa(E) := \sup_{\mu \in \cM_1(E)} \exp\left(-\int_{E}\int_{E} \log\left(|x-y|^{-1}\right) d\mu(x)d\mu(y) \right).
\end{align*}
The logarithmic capacity of a finite-gap compact $E$ can be determined as the limit of the geometric means of the off-diagonal sequences of Jacobi operators in the isospectral torus \cite{SIM11, SY97}: 
\begin{align}
\label{eq:captoa}
\capa(E) = \lim_{n \to \infty} (a_1a_2\cdots a_n)^{1/n}, \;\; J(a,b) \in \cJ(E).
\end{align}
We denote by $\cK_\alpha^n$ the set of $n$-gap compact subsets of $\bbR$ of capacity $\alpha$:
\begin{align*}
\cK_\alpha^n := \{E \subset \bbR: E \;\;n\text{-gap compact},\; \capa(E) = \alpha\}.
\end{align*}
Similarly, we denote $\cK^n$ the set of all $n$-gap real compacts, and by $\cK_\alpha$ the set of finite-gap compact subsets of $\bbR$ of capacity $\alpha$:
\begin{align*}
\cK^n = \bigcup_{\alpha > 0} \cK^n_\alpha, \;\;\;
\cK_\alpha = \bigcup_{n=1}^\infty \cK^n_\alpha.
\end{align*}
Note that, by \eqref{eq:captoa}, if $E = \sigma(H_V)$ for some DSO $H_V$, then $E \notin \cK_\alpha$ for any $\alpha \neq 1$.

Comparing equations \eqref{eq:dso}, \eqref{eq:jacobi}, and \eqref{eq:captoa} immediately reveals that DSOs are a special class of Jacobi operator with spectrum $E$ of unit capacity $\capa(E) = 1$.  Thus, an absolutely continuous, almost-periodic DSO with finite-gap spectrum $E$ is a member of the isospectral torus $\cJ(E)$.  However, the isospectral torus is defined in terms of the much broader class of Jacobi operators.  In this context, the following question is quite natural: 
\begin{qstn}
Consider $E \in \cK_1$.  Does the isospectral torus $\cJ(E)$ contain a DSO?
\end{qstn}
Our titular result says that the answer to this question is generically ``no":

\begin{thm}
\label{thm:specatyp}
The set of finite-gap compacts $E \in \cK_1$ for which there exists a DSO $H \in \cJ(E)$ is meager and has measure zero.
\end{thm}
By a meager set, we mean a countable union of nowhere dense sets in the relevant topology.  Here, the topology and measure on $\cK_1$ is induced by a topology on each $\cK_\alpha^n$; namely, any $E \in \cK^n$ can be identified uniquely by its ordered $2(n+1)$ distinct gap edges.  The topology and measure are then induced by inclusion of the corresponding vector in $\bbR^{2(n+1)}$.  We clarify precisely what we mean here in Section 2.
\begin{rem}
The restriction to $E \in \cK_1$ is necessary here to make the theorem statement nontrivial; indeed, we will see later that $\cK^n_1$ is codimension $1$ inside $\cK^n$.  However, if we were to extend our definition of DSO to include any Jacobi operator with constant off-diagonal sequence $a$, all of our theorems would remain true without any restriction on the capacity of $E$.

We also would be remiss not to mention the recent work \cite{HUR17}, wherein the methods of canonical systems are used to prove a result on the sparsity of $m$-functions of DSOs amongst those for Jacobi operators.
\end{rem}

On the other hand, it is known in general that, if $V(\cdot) = V(\cdot+p)$ is periodic for some $p \in \bbZ$, the associated DSO $H_V$ is purely absolutely continuous with spectrum consisting of a union of at most $p$ closed intervals.  Our second result says that, at least when the number of gaps in $E$ is small, a sort of converse holds:
\begin{thm}
\label{thm:lowgapper1}
Suppose there exists an almost-periodic DSO $H$ with $\sigma(H) = E$ with $n= 0, 1,$ or $2$ open spectral gaps such that the essential support of the a.c. spectrum $\Sigma_{ac}(H)$ agrees with $E$ up to sets of Lebesgue measure zero; that is, $|E \setminus \Sigma_{ac}(H)| = 0$.  Then $H$ is periodic.
\end{thm}
That this is true for the case with $n = 0$ gaps has been explored in-depth and is quite well-known \cite{BOR46, DEN04, HOC75, KS03}.  Similarly, while we have never seen the case $n = 1$ explicitly addressed, it has probably been observed before by virtue of trace formulas.  The real novelty of this theorem is that the same holds when $n = 2$, which we will see is the first truly nontrivial case.  Of course, this raises the following natural follow-up question: are there any aperiodic absolutely continuous DSOs with finite-gap spectrum? While Theorem \ref{thm:lowgapper1} is evidence that the answer could be negative, we are reluctant to make a conjecture one way or another.  Theorem \ref{thm:lowgapper1} is thoroughly algebro-geometric in nature, and it may be that the result is a consequence of the low dimension of the corresponding isospectral torus rather than an illustration of a universal principle.

Theorems \ref{thm:specatyp} and \ref{thm:lowgapper1} are more precisely viewed as corollaries of broader statements which can be made in the context of reflectionless Jacobi operators, which we now define.  Denote by $r(n,z,J) =  \langle \delta_n, (J-z)^{-1}\delta_n \rangle$ the Green's function for $J$.  For each $n \in \bbZ$ the Green's function $r(n,z,J)$ has almost-everywhere defined radial limits on $\sigma(J)$; that is, for almost-every $x \in \sigma(J)$, $\lim_{\e \downarrow 0} r(n,x + i\e,J)$ exists.  When this limit exists, we denote it by $r(n,x+i0,J)$.

Consider a positive-measure compact $E \subset \bbR$.  We say that the Jacobi operator $J$ is reflectionless on $E$ when $\Re(r(n,x+i0,J)) = 0$ for (Lebesgue) almost-every $x \in E$ for all $n \in \bbZ$.  A Jacobi operator is called reflectionless if it is reflectionless on its spectrum $\sigma(J)$.  There is a subtle and intimate relationship between the reflectionless property, almost periodicity, and the geometry and absolute continuity of the spectrum \cite{AVI15, BRS10, GMZ08, GY06, GZ09, KOT84, REM11, SY97, VY14}.  From a heuristic perspective, reflectionless Jacobi operators are those whose dynamics lack reflection; that is to say, those operators whose states coming completely from $-\infty$ in backward time proceed completely to $+\infty$ in forward time \cite{BRS10}.  Periodic Jacobi operators are perhaps the initial example of reflectionless operators, but, more generally, almost-periodic Jacobi operators are reflectionless on the essential support of their a.c. spectrum \cite{KOT84, REM11}.  A partial converse likewise holds under certain geometric restrictions on the spectrum \cite{SY97}; however, the converse may fail in spectacular fashion without these restrictions \cite{AVI15, VY14}.  

Fix a positive-measure compact $E \subset \bbR$, and define the isospectral torus for $E$ as
\begin{align*}
\cJ(E) := \{ J : \sigma(J) \subset E \text{ and } J \text{ reflectionless on } E\}.
\end{align*}
This is a compact, path-connected subset of the space of uniformly bounded Jacobi operators when endowed with the topology of pointwise convergence \cite{PR11, REM13}.   The action $\cS$ of conjugation by the left-shift $S : \delta_n \mapsto \delta_{n+1}$ preserves both the spectrum and the reflectionless condition, and thus $(\cJ(E), \cS)$ is a discrete-time dynamical system.  Theorem \ref{thm:specatyp} is a straightforward consequence of the following dynamical characterization of those isospectral tori containing a DSO:
\begin{thm}
\label{thm:notminimal}
Fix a positive-measure compact $E \subset \bbR$, and suppose there exists a DSO $H_V \in \cJ(E)$.  Then either $V$ is a constant potential $V = C$, or the dynamical system $(\cJ(E), \cS)$ is not minimal.
\end{thm}
\begin{rem}
Note that Theorem \ref{thm:notminimal} makes no assumption on $E$ beyond positive-measure and compact; in particular, the result applies to sets with infinitely many spectral gaps.  Furthermore, Theorem \ref{thm:notminimal} does not merely hold for the topology of pointwise convergence on $\cJ(E)$; it holds for any topology under which the limit of a DSO remains a DSO.
\end{rem}

The dichotomy in Theorem \ref{thm:notminimal} raises a follow-up question of its own: under the assumption that there exists a DSO $H \in \cJ(E)$, to what extent can the shift dynamics fail to be minimal? While we do not have a complete answer to this question, we can report the following progress:
\begin{thm}
\label{thm:lowgapper2}
Suppose $E \in \cK_1^n$ for $n = 0, 1,$ or $2$.  If there exists a DSO $H \in \cJ(E)$, the dynamical system $(\cJ(E), \cS)$ is periodic.
\end{thm}

Of course, without the assumptions of finite-gap and reflectionless, the conclusion of Theorem \ref{thm:lowgapper2} fails completely: one can construct limit-periodic DSOs with pure point spectrum and finitely many gaps \cite{POS83}, and the subcritical AMO and Pastur-Tkachenko class of DSOs are reflectionless with homogeneous infinite-gap spectra \cite{DGSV15, FLta}.  But it would be very interesting to construct an example of a reflectionless Jacobi operator with periodic off-diagonal, aperiodic main diagonal, and finite-gap spectrum.

This paper will be structured as follows: Section 2 contains a brief crash-course on the requisite background material.  Section 3 contains the majority of the proofs of the above theorems, save for a proposition demonstrating the genericity with which Theorem \ref{thm:specatyp} holds.  This proposition (which may be well-known, although we could not find it in the literature) is proven in Section 4.  Section 5 can be viewed as an appendix; it includes some basic facts about the Toda hierarchy that we implement in our proofs.

\begin{ack}
We would like to thank Daniel Bernazzani and Vitalii Gerbuz for a number of helpful conversations, and offer a special thanks to Jake Fillman for pointing out the application of these methods to off-diagonal Jacobi operators.
\end{ack}

\section{Background}

A Jacobi operator is a linear operator $J$ on sequences $u \in \bbC^\bbZ$ parametrized by a pair of bounded, real-valued sequences $a, b \in \ell^\infty(\bbZ)$.   The operator is defined termwise by
\begin{align*}
(Ju)_n = a_nu_{n-1} + b_nu_n + a_{n+1}u_{n+1}.
\end{align*}
In this paper, we will deal only with those Jacobi operators whose off-diagonal sequences $a$ have entries which are positive and bounded uniformly away from $0$.   We make no notational distinction between $J$ and its restriction to $\ell^2(\bbZ)$.

The resolvent set $\rho(J)$ of a Jacobi operator $J$ is defined as the set of complex energies $z$ for which the operator $(J-zI)^{-1}$ exists and is bounded on $\ell^2(\bbZ)$.  The complement of the resolvent set in the plane is called the spectrum of $J$, denoted $\sigma(J) = \bbC \setminus \rho(J)$.  Because $J$ is bounded and self-adjoint, $\sigma(J)$ is a compact subset of $\bbR$.  The absolutely continuous spectrum $\sigma_{ac}(J)$ is the common topological support of the absolutely continuous parts of all spectral measures for $J$.  

Define the Green's function
\begin{align*}
r(n,z,J) := \langle \delta_n, (J-zI)^{-1}\delta_n \rangle, \;\; z \in \rho(J), n \in \bbZ.
\end{align*}  
For each $n \in \bbZ$, the Green's function $r(n,z,J)$ is Herglotz; that is, a function which is analytic from the upper half plane to itself.  As such, $r(n,z,J)$ has Lebesgue almost-everywhere defined radial limits on $\sigma(J)$: for almost-every $x \in \sigma(J)$, $\lim_{\e \downarrow 0} r(n,x + i\e,J) =: r(n,x+i0,J)$ exists in $\bbC \cup \{\infty\}$.  We say $J$ is reflectionless on a positive-measure set $E$ when
\begin{align*}
\Re(r(n,x+i0,J)) = 0
\end{align*}
for (Lebesgue) almost-every $x \in E$ for every $n \in \bbZ$.   If $J$ is reflectionless on its spectrum $\sigma(J)$, we simply call $J$ reflectionless.

Recall also that we define the isospectral torus for $E$ as
\begin{align*}
\cJ(E) := \{ J : \sigma(J) \subset E \text{ and } J \text{ reflectionless on } E\}.
\end{align*}
We endow $\cJ(E)$ with the topology of pointwise convergence of the parametrizing sequences $(a,b)$.  Conjugation by the left-shift $S : \delta_n \mapsto \delta_{n+1}$ preserves the reflectionless condition and the spectrum: if $J \in \cJ(E)$, so is $\cS J := SJS^{-1}$.

Almost-periodic Jacobi operators $J$ are reflectionless on the essential support $\Sigma_{ac}(J)$ of their a.c. spectrum:
\begin{thm}[Theorem 1.4, \cite{REM11}]
\label{thm:rem}
An almost-periodic Jacobi operator $J$ is reflectionless on $\Sigma_{ac}(J)$.
\end{thm}
Under certain geometric restrictions on the spectrum, one can likewise conclude almost-periodicity from reflectionlessness:
\begin{thm}[Theorem C, \cite{SY97}]
\label{thm:sy}
Suppose $E \in \cK^n$.  Then the dynamical system $(\cJ(E), \cS)$ is conjugate to $((\bbR/\bbZ)^n, T_\omega)$, where $T_\omega(\alpha) = \alpha + \omega$, and $\omega$ is given by
\begin{align*}
\omega_j = \omega_E((-\infty, E_j^-]).
\end{align*}
Furthermore, every $J \in \cJ(E)$ is almost-periodic and purely absolutely continuous.
\end{thm}
The values $\omega_j$ are called the harmonic frequencies of $E$.  Here, $\omega_E$ is the harmonic measure for the domain $\Omega = \hat{\bbC} \setminus E$ with pole at infinity.  

Consider the set $\cM_1(E)$ of Borel probability measures supported on $E$.  For $\mu \in \cM_1(E)$, define its potential energy $\cE(\mu)$ as
\begin{align*}
\cE(\mu) = \int_{E}\int_{E} \log\left(|x-y|^{-1}\right) d\mu(x)d\mu(y)
\end{align*}
If there exists $\mu \in \cM_1(E)$ for which $\cE(\mu) < \infty$, then $\cE(\mu)$ has a unique minimizer called the equilibrium measure $d\rho_E$.  When $E$ is finite-gap, the harmonic measure $\omega_E(dx)$ agrees with the equilibrium measure \cite[Theorem 4.3.14]{RAN95}.  In this sense, we define the logarithmic capacity of $E$
\begin{align*}
\capa(E) &:= \sup_{\mu \in \cM_1(E)} \exp(-\cE(\mu)) \\
&= \exp(-\cE(\omega_E(dx))).
\end{align*}
The capacity scales linearly; that is, if $E$ is a compact set, $\alpha > 0$, and $\alpha E$ is the $\alpha$-rescaling of $E$, one has \cite[Theorem 5.1.2]{RAN95}
\begin{align}
\label{eq:capscale}
\capa(\alpha E) = \alpha \capa(E), \;\; \alpha > 0.
\end{align}

The genericity of our theorems holds in the following sense: fix a finite union of nondegenerate closed intervals $E \in \cK^n$, and write it as a closed interval with a finite union of maximal open gaps removed:
\begin{align}
\label{eq:Ecpct}
E = [E_0^+,E_0^-] \setminus \bigcup_{j =1}^n (E_j^-, E_j^+).
\end{align}
To avoid degenerate spectral bands, we assume $E_{j}^+ < E_{j+1}^- $ for each $0 \leq j \leq n$, with the convention $E_{n+1}^- := E_0^-$.  For $n \in \bbN$, denote
\begin{align*}
T^n = \{(x_1, \cdots , x_{n}) \in \bbR^n : x_1 < x_2 < \cdots < x_{n} \}.
\end{align*}  
Any compact $E \in \cK^n$ written as \eqref{eq:Ecpct} can be parametrized by a vector in $T^{2(n+1)}$ by reading off endpoints from left to right:
\begin{align*}
E \mapsto (E_0^+, E_1^-, E_1^+, E_2^-,\cdots, E_n^+, E_0^-)^\top.
\end{align*}
This parametrization naturally induces a topology and measure on $\cK^n$ via the subspace topology of $T^{2(n+1)} \subset \bbR^{2(n+1)}$.

The topology on $\cK^n$ induces a projective topology on $\cK^n_\alpha$ in the following way.  By \eqref{eq:capscale}, we have that, with the identification $\cK^n$ with $T^{2(n+1)}$ above, we can topologize $\cK_\alpha^n$ with the subspace topology and pushforward measure from the projection from $\cK^n \cong \cK^n_\alpha \times \bbR_+$ (see Proposition \ref{pr:homeo} below).  With these conventions, we have the following:
\begin{prop}
\label{pr:genratind}
For each positive integer $n$ and real $\alpha > 0$, the set of $E \in \cK_\alpha^n$ with rationally dependent harmonic frequencies is meager and has Lebesgue measure zero.
\end{prop}
Though we could not explicitly find a result of this nature in the literature, this result could be concluded as a consequence of previous work (e.g., \cite{TOT09}).  We offer a self-contained proof of this proposition in Section 4.  

As a consequence of Theorem \ref{thm:sy} and Proposition \ref{pr:genratind}, in the finite-gap regime one generically has minimality of the shift on the isospectral torus.  In particular, it is typically the case that the hull of a given almost-periodic $J \in \cJ(E)$ is the entire isospectral torus.  To prove our theorems, we utilize this fact in conjunction with a family of integral curves on $\cJ(E)$ called the Toda Hierarchy.

Consider a bounded linear operator $A$ on $\ell^2(\bbZ)$.  Denote by $A^\pm$ the restrictions of $A$ to $\ell^2(\bbZ_\pm) \hookrightarrow \ell^2(\bbZ)$, where the inclusion map is given by assigning zeros to the left- or right-half line.  Fix a polynomial $P$ of degree $n+1 \geq 1$.  The $n^{th}$ Toda flow (for $P$) is the integral curve $J(t)$ of Jacobi operators satisfying the Lax pair
\begin{align}
\label{eq:laxpair}
\partial_t J &= [P(J)^+ - P(J)^-, J].
\end{align} 
There exist unique solutions to \eqref{eq:laxpair} for any bounded Jacobi initial condition $J_0$ \cite[Theorem 12.6]{TES00}.  When there exists a monic polynomial $P$ so that
\begin{align}
\label{eq:stationary}
[P(J)^+ - P(J)^-, J] = 0.
\end{align}
we say $J$ is stationary for $P$.

For particular choices of polynomial $P$, the Toda flow induces a system of differential equations on the parametrizing sequences $a, b \in \ell^\infty$.  The critical facts about the Toda flow that we employ are summarized in
\begin{prop}
\label{pr:rem1}
For any non-constant polynomial $P$:
\begin{enumerate}
\item \cite[Corollary 1.3]{PR11} Suppose $J(t)$ is the unique solution to \eqref{eq:laxpair} with $J(0) = J_0 \in \cJ(E)$, where $E = \sigma(J_0)$.  Then $J(t) \in \cJ(E)$ for all $t \in \bbR$.
\item \cite[Theorem 12.8]{TES00} The stationary solutions $\partial_t J = 0$ of \eqref{eq:laxpair} are finite-gap reflectionless Jacobi operators.
\item \cite[Corollary 12.10]{TES00} If $J$ is stationary for $P$, then the spectrum of $J$ has at most $\deg(P) - 1$ (open) spectral gaps.  Furthermore, for each $E \in \cK^n$, there exists a polynomial $P$ of degree $n+1$ so that every $J \in \cJ(E)$ is stationary for $P$.
\end{enumerate}
\end{prop}

\begin{rem}
We note that property \textit{(3)} above is not stated verbatim from its cited source; we offer its proof in Section 5.
\end{rem}

\section{Proofs of the Main Theorems}

We now have all of the requisite tools to prove our theorems.

\begin{proof}[Proof of Theorem \ref{thm:notminimal}]
Consider the Toda lattice, given by \eqref{eq:laxpair} with $P(z) = z$.  The associated system of differential equations is
\begin{align}
\label{eq:tl1}
\partial_t a_n &= a_n(b_{n+1} - b_n), \\
\label{eq:tl2}
\partial_t b_n &= 2(a_n^2 - a_{n-1}^2).
\end{align}
Suppose there exists a DSO $H \in \cJ(E)$ and that $(\cJ(E), \cS)$ is minimal.  Then, by minimality of the shift and the topology on $\cJ(E)$, every $J = J(a,b) \in \cJ(E)$ has $a_n = 1$ for all $n$.  Consider an arbitrary $J_0 \in \cJ(E)$, and let $J(t)$ solve \eqref{eq:laxpair} with $J(0) = J_0$.  By Proposition \ref{pr:rem1}, $J(t) \in \cJ(E)$ for all $t$, and thus $a_n(t) = 1$ for all $t \in \bbR$ and all $n \in \bbZ$.  But then $\partial_t a_n = 0$, and by \eqref{eq:tl1} $b_{n+1} = b_n$ for all $n \in \bbZ$.  
\end{proof}
From here, the proof of Theorem \ref{thm:specatyp} is a matter of some bookkeeping:
\begin{proof}[Proof of Theorem \ref{thm:specatyp}]
Fix an $E \in \cK^n_1$, $n \geq 1$, and suppose there exists a DSO $H \in \cJ(E)$.  Since $E$ has at least one spectral gap, $H = H_V$ does not have constant potential $V$, so the dynamics $(\cJ(E), \cS)$ are not minimal by Theorem \ref{thm:notminimal}.  But by Theorem \ref{thm:sy}, the dynamical system $(\cJ(E), \cS)$ conjugates to $((\bbR/\bbZ)^n, T_\omega)$.  Since $(\cJ(E), \cS)$ is not minimal, some entries of $\omega$ are rationally dependent.  This condition is meager and measure zero by Proposition \ref{pr:genratind}.
\end{proof}
We extrapolate the ideas used above to prove Theorems \ref{thm:lowgapper1} and \ref{thm:lowgapper2}.  

\begin{proof}[Proof of Theorem \ref{thm:lowgapper2}]
Fix a compact $E$ with zero, one, or two open spectral gaps.  We prove by case analysis on the number of gaps:
\begin{enumerate}
\item If $E$ has no spectral gaps, $\cJ(E)$ is a singleton set.  If $\cJ(E)$ contains a DSO $H_V$, $H_V$ must be stationary for $\eqref{eq:tl1}$ with $a_n = 1$ for all $n$; in particular, $H_V = S+S^{-1} + C$ is $1$-periodic.
\item Suppose $E$ has one spectral gap.  Then if $\cJ(E)$ contains a DSO $H_V$, $V$ cannot be constant, and the shift dynamics thus cannot be minimal by Theorem \ref{thm:notminimal}.  Thus, the lone harmonic frequency must be rational, and by Theorem \ref{thm:sy}, the dynamical system $(\cJ(E), \cS)$ is periodic, as claimed.
\item Suppose $E$ has two spectral gaps, and suppose there is a DSO $H_V \in \cJ(E)$.  Then by Proposition \ref{pr:rem1}, $H_V$ is stationary for some monic degree 3 polynomial $P(z) = z^3 + c_1z^2+c_2z$. Noting that \eqref{eq:stationary} is linear in the choice of polynomial $P$, it is a long but straightforward calculation (cf. \cite[Equation (12.44)]{TES00}) to see that, from $\partial_t b_n = 0$ and $a_n = 1$, we have
\begin{align*}
(b_{n+1} - b_{n-1})(b_{n+1} + b_n + b_{n-1}) = -c_1(b_{n+1} - b_{n-1}).
\end{align*}
Thus, for some constant $C = -c_1$, at least one of
\begin{align}
\label{eq:case1}
b_{n+1} + b_n + b_{n-1} &= C \\
\label{eq:case2}
b_{n+1} &= b_{n-1}
\end{align}
must hold for every $n \in \bbZ$.  We claim that, if one of \eqref{eq:case1} or \eqref{eq:case2} holds for $b_n$ for all $n \in \bbZ$, then $b_n$ can assume only finitely many distinct values.

If $b_{n+1} = b_{n-1}$ for every $n$, $b$ is 2-periodic and we are done.  Otherwise, there exists some $n_0$ for which $b_{n_0+1} \neq b_{n_0-1}$.  By shifting a finite number of times, we may assume this $n_0 = 1$.  Thus, we have
\begin{align}
\label{eq:b2}
b_2 = C - (b_0 + b_1).
\end{align}
We proceed inductively.  Denote by 
\begin{align*}
\cB :=  \{ (b_i, b_j), 0 \leq i \neq j \leq 2 \}
\end{align*}
the set of all possible distinct pairs taken from $\{b_0, b_1, b_2\}$, listed with multiplicity.  Here, we have passed to considering distinct pairs $(b_i,b_j)$ because it is more challenging to proceed inductively via our relations \eqref{eq:case1}, \eqref{eq:b2} when multiplicities are not taken into account.

Suppose that the pair $(b_n, b_{n-1}) \in \cB$, and consider the pair $(b_{n+1}, b_n)$.  If $\eqref{eq:case2}$ holds, we are clearly done by assumption; otherwise, $b_{n+1} = C - (b_n + b_{n-1})$.  But by our assumption that $(b_n, b_{n-1}) \in \cB$, we have $b_n + b_{n-1} \in \{C - b_0, C - b_1, b_0 + b_1\}$.  In any case,  $b_{n+1} = C - (b_n + b_{n-1})$ implies $b_{n+1} \in \{b_0, b_1, b_2\} \setminus \{b_n\}$, and hence $(b_{n+1}, b_n) \in \cB$.  This shows that $b_n$ takes finitely many values for $n \geq 0$.  One can clearly apply the identical argument in reverse by the symmetry of \eqref{eq:case1} and \eqref{eq:case2} to conclude the same for $n < 0$.

Consequently, one sees that, if $(1,b)$ is stationary, $b_n$ can assume only finitely many values.  But because $H_V \in \cJ(E)$ and $E$ is finite-gap, $b_n$ is an almost-periodic sequence by Theorem \ref{thm:sy}.  An almost-periodic sequence taking finitely many values is periodic, so $H_V \in \cJ(E)$ is periodic.
\end{enumerate}
So, in each case, the existence of a DSO $H \in \cJ(E)$ implies that $H$ is periodic.  Thus, the compact $E$ must be of periodic type; that is, $\omega_j \in \bbQ$ for each $j = 1, \cdots, n$ (see, e.g., \cite[Corollary 5.5.19]{SIM11}).  By Theorem \ref{thm:sy}, it follows that the dynamical system $(\cJ(E), \cS)$ is periodic.
\end{proof}

Of course, this makes the proof of Theorem \ref{thm:lowgapper1} quite simple:

\begin{proof}[Proof of Theorem \ref{thm:lowgapper1}]
Fix a finite-gap compact $E$ with $n = 0, 1, $ or $2$ gaps, and suppose there exists an almost-periodic DSO $H$ with $\sigma(H) = E$ such that the essential support of the a.c. spectrum $\Sigma_{ac}(H) = E$ agrees with $E$ up to sets of Lebesgue measure zero.  Then by Theorem \ref{thm:rem}, $H \in \cJ(E)$, and by Theorem \ref{thm:lowgapper2}, $H$ is periodic.
\end{proof}

\begin{rem}
As noted in the Acknowledgments, we are thankful to Dr. Jake Fillman for pointing out that some of our methods extend beyond the class of DSOs.  Namely, we call $J_0 = J_0(a)$ an off-diagonal Jacobi operator (ODJO) if it is of the form
\begin{align*}
(J_0u)_n = a_n u_{n-1} + a_{n+1}u_{n+1}.
\end{align*}
Theorem \ref{thm:notminimal} holds with ``DSO", ``$H_V$", and ``$V$" replaced by ``ODJO", ``$J_0(a)$", and ``$a$", respectively, by noting that \eqref{eq:tl2} must be zero under the assumption of minimality of $(\cJ(E), \cS)$.  Consequently, the natural analogue of Theorem \ref{thm:specatyp} likewise holds for the ODJO class.  However, our methods do not immediately extend to the proofs of Theorems \ref{thm:lowgapper1} and \ref{thm:lowgapper2}, because the assumption that $b_n = 0$ leads to certain degeneracies in the Toda hierarchy.
\end{rem}

\section{Frequencies of Finite-gap Compacts are Generically Independent}

In this section, we prove Proposition \ref{pr:genratind}.

Fix a compact $E \in \cK^n$, and write it as a closed interval without a finite union of maximal open gaps as in \eqref{eq:Ecpct}:
\begin{align*}
E = [E_0^+,E_0^-] \setminus \bigcup_{j =1}^n (E_j^-, E_j^+).
\end{align*}
Associated to the compact $E$ is a degree $2(n+1)$ polynomial $Q$ given by
\begin{align}
\label{eq:qe}
Q_E(z) = \prod_{j=0}^n (z-E_j^+)(z-E_j^-)
\end{align}
It is not hard to see that this polynomial is positive in $\bbR \setminus E$.

The following is classical, but we reproduce a proof (found in this form in \cite{SSW01}) for the sake of completeness:
\begin{prop}
\label{pr:lindep}
Fix a compact $E \in \cK^n$.  There exists a unique monic critical polynomial $P_E(z)$ of degree $n$ so that
\begin{align}
\label{eq:critpoly}
\int_{E_j^-}^{E_j^+} \frac{P_E(x)}{\sqrt{Q_E(x)}}dx = 0, \; \; j = 1, 2, ..., n.
\end{align}
If we write
\begin{align*}
P_E(z) = z^n - c_1z^{n-1} - ... - c_{n-1}z - c_n,
\end{align*}
then the vector $\vec{c}$ with $(\vec{c})_j = c_j$ is the unique solution to the linear system $A\vec{c} = \vec{b}$, where
\begin{align*}
(A)_{jk} &= \int_{E_j^-}^{E_j^+} \frac{x^{n-k}}{\sqrt{Q_E(x)}}dx \\
(\vec{b})_j &= \int_{E_j^-}^{E_j^+} \frac{x^n}{\sqrt{Q_E(x)}}dx.
\end{align*}
\end{prop}

\begin{proof}
By definition, if there exists a unique $\vec{c}$ solving $A\vec{c} = \vec{b}$ as claimed, then the critical polynomial exists.  Thus, it suffices to show the linear equation has a unique solution, i.e. that the matrix $A$ is non-singular.

Suppose otherwise.  Then there exists a nonzero polynomial $P$ of degree at most $n-1$ so that \eqref{eq:critpoly} holds.  Since $Q$ is positive on each gap $(E_j^-, E_j^+)$, $P$ must change sign on the interior of each gap $(E_j^-, E_j^+)$, $j = 1, 2, ..., n$.  But then $P$ has at least $n$ zeroes, contradicting our assumption on its degree.
\end{proof}

Via a straightforward calculus exercise, it isn't hard to check that

\begin{lem}
\label{lem:intanalytic}
The value
\begin{align*}
I_k(a) := \int_0^a \frac{x^k}{\sqrt{x(a-x)}}dx
\end{align*}
is given by
\begin{align*}
I_k(a) &= \frac{\Gamma(k+1/2)}{\Gamma(k+1)}\sqrt{\pi} a^k
\end{align*}
where $\Gamma$ is the typical gamma function.  In particular, $I_k(a)$ is analytic in $a$.
\end{lem}

With this in hand, one has that

\begin{lem}
\label{lem:coeffsanalytic}
The coefficients $c_j$ above depend analytically on the gap edges of $E$.
\end{lem}
\begin{proof}
By nonsingularity of the matrix $A$, if we show that terms like
\begin{align}
\label{eq:moment}
m_{j}(E) := \int_{E_j^-}^{E_j^+} \frac{x^n}{\sqrt{Q_E(x)}}dx
\end{align}
are analytic in each $E_k^\pm$, we are done, because the $c_j$ are linear combinations of entries of $A^{-1}$ and $\vec{b}$ above.  Since $Q_E$ is positive and bounded away from zero on $(E_j^-, E_j^+)$ when varying $E_k^\pm$, $k \neq j$, \eqref{eq:moment} is clearly analytic for all except $E_j^\pm$.  By symmetry in $\pm$, it suffices to prove that the terms \eqref{eq:moment} are analytic in $E_j^+$.

Fix a positive integer $n \geq 1$, and denote by $U$ the following open subset of $\bbR^{2(n+1)+1}$:
\begin{align*}
U = \left\{ (v, t) \in T^{2(n+1)} \times \bbR : t \in \bigcup_{j=1}^{n}(v_{2j-1}, v_{2j}) \right\}
\end{align*}
Define a function $f_j : U \to \bbR$ by
\begin{align*}
f_j(E,x) := \frac{x^n}{\sqrt{\left|\prod_{k \neq j}(E_k^+ - x)(E_k^- - x) \right|}}.
\end{align*}
This function is analytic for $(E, x) \in B_\delta(E) \times (E_j^- - \delta, E_j^+ + \delta)$ for $\delta$ smaller than the shortest band and gap lengths.  In this region, we expand
\begin{align*}
f_j(E,x) &= \sum_{k = 0}^\infty a_k^j(E)x^k.
\end{align*}
where the coefficients $a_k^j(E)$ are analytic in each $E_i^\pm$, $i \neq j$.  Then we have
\begin{align*}
m_{j}(E) &= \int_{E_j^-}^{E_j^+} \frac{f_j(E,x)}{\sqrt{(E_j^+-x)(x-E_j^-)}}dx \\
&= \sum_{k = 1}^\infty a_k^j(E)\int_{E_j^-}^{E_j^+} \frac{x^k}{\sqrt{(E_j^+ -x)(E_j^- - x)}} dx \\
&= \sum_{\vec{m} \in [1,2(n+1)]^n} c_{\vec{m}}E^{\vec{m}},
\end{align*}
with multi-index notation in the last equality.  Above, the first equality is by definition, the second by analytic expansion of $f_j$ and the Fubini theorem, and the last by Lemma \ref{lem:intanalytic} and analyticity of the coefficients $a_k^j(E)$.
\end{proof}

Thus, for each compact $E \subset \bbR$, we find an associated probability measure given by
\begin{align}
\label{eq:eqmeas}
d\rho_E(x) = \frac{1}{\pi}\sum_{j=0}^n \frac{|P_E(x)|}{\sqrt{|Q_E(x)|}}\chi_{[E_j^+,E_{j+1}^-]}(x)dx
\end{align}
where here we interpret $E_{n+1}^- := E_0^-$.  In fact, \eqref{eq:eqmeas} is the equilibrium measure for $E$ \cite{WID69}; that is, $d\rho_E$ is the unique probability measure $\mu$ supported on $E$ minimizing
\begin{align*}
\cE(\mu) = \int_E \int_E \log(|x-y|^{-1})d\mu(x)d\mu(y).
\end{align*} 
For finite-gap sets of the form \eqref{eq:Ecpct} with non-degenerate bands, recall that the equilibrium measure agrees with harmonic measure (e.g., \cite[Theorem 4.3.14]{RAN95}).  Thus, we may equivalently define the frequency vector $\omega \in \bbR^n$ from Theorem \ref{thm:sy} by
\begin{align*}
\omega_j = \rho_E((-\infty, E_j^-]), \;\; j = 1, 2, \cdots, n.
\end{align*}
Because the pushforward of harmonic measure via a conformal bijection is the harmonic measure, the frequency vector $\omega \in \bbR^n$ is scaling-invariant; that is, 
\begin{align}
\label{eq:freqscale}
\omega(\alpha E) = \omega(E), \;\; \alpha > 0.
\end{align}
One could likewise check this fact by a straightforward calculation using Proposition \ref{pr:lindep} and \eqref{eq:eqmeas}.

Fix a number $n \in \bbN$, and denote
\begin{align*}
T^n = \{(x_1, \cdots , x_{n}) \in \bbR^n : x_1 < x_2 < \cdots < x_{n} \}.
\end{align*}  
Any compact $E \in \cK^n$ written \eqref{eq:Ecpct} can be parametrized by a vector in $T^{2(n+1)}$ by reading off endpoints from left to right:
\begin{align*}
E \mapsto (E_0^+, E_1^-, E_1^+, E_2^-,\cdots, E_n^+, E_0^-)^\top
\end{align*}
We identify $\cK^n$ to $T^{2(n+1)}$ in this way, and note that $\cK^n$ can be naturally viewed as a $\bbR_+$ bundle over $\cK_1^n$:
\begin{prop}
\label{pr:homeo}
Fix an $\alpha > 0$.  The space $\cK^n$ is homeomorphic to $\cK^n_\alpha \times \bbR_+$.
\end{prop}
\begin{proof}
Consider the map $h_\alpha : \cK^n \to \cK^n_\alpha \times \bbR_+$ given by
\begin{align*}
h_\alpha(E) := \left( \frac{\alpha}{\capa(E)} E, \;\;\capa(E) \right).
\end{align*}
By \eqref{eq:capscale}, $h_\alpha$ is a bijection, with continuous inverse
\begin{align*}
h_\alpha^{-1}(E, \beta) = \frac{\beta}{\alpha} E.
\end{align*}
The capacity $\capa(E)$ is continuous and positive on $\cK^n$ by Proposition \ref{pr:lindep} and \eqref{eq:eqmeas}.  Thus, $h_\alpha$ is likewise continuous, and $h_\alpha$ is a homeomorphism, as claimed.
\end{proof}

Define a map $\omega : \cK^n \to T^n$ which outputs the frequency vector $\omega(E) \in T^n$ of a compact $E \in \cK^n$.   By \eqref{eq:freqscale}, this map is constant on the $\bbR_+$ fibres over $\cK_\alpha^n$.  We wish to investigate the ``typicality" of rational independence of the frequency vector with respect to the input compact $E$.  With this in mind, we explore some properties of the map $\omega$:

\begin{prop}
\label{pr:omcts}
The map $\omega: \cK^n \to T^n$ is real analytic.  Furthermore, if $B \subset \cK^n$ has measure zero, $\omega^{-1}(B)$ likewise has measure zero.
\end{prop}

\begin{proof}
We begin by proving analyticity.

Recall the open set $U$ defined above, and denote by $g_j(E,t) : U \to \bbR$ the function
\begin{align*}
g_j(E,t) = \frac{|P_E(t)|}{\sqrt{\left|\prod_{k \neq j} (t - E_{k-1}^+)(t-E_k^-) \right|}}
\end{align*}
By Lemma \ref{lem:coeffsanalytic} and because the zeros of $P_E$ lie in the gaps of $E$, for each $E \in \cK^n$ there exists a value $\delta > 0$ so that $g_j(E, t)$ is bounded and analytic on $B_\delta(E) \times (E_{j-1}^+-\delta, E_j^- +\delta)$.  Notice that, as defined, we have
\begin{align*}
\omega(E)_j = \frac{1}{\pi}\sum_{k \leq j} \int_{E_{k-1}^+}^{E_k^-} \frac{g_k(E,t)}{\sqrt{(t-E_{k-1}^+)(E_k^- - t)}}dt.
\end{align*}
Each term in this sum is analytic by an argument exactly analogous to that in the proof of Lemma \ref{lem:coeffsanalytic} above.  Consequently, $\omega : \cK^n \to T^n$ is analytic, as claimed.

Consider now the Jacobian $D\omega$ of $\omega$. By Sard's theorem, the critical set $X \subset \cK^n$ of values for which the Jacobian is not surjective has zero measure image under $\omega$, that is, $|\omega(X)| = 0$.

Define a new analytic function $g : \cK^n \to \bbR$ given by
\begin{align*}
g(E) = \sum_{\vec{m} \in [1, 2(n+1)]^n} \det((D\omega(E))_{\vec{m}})^2
\end{align*}
where $(D\omega)_{\vec{m}}$ is the matrix formed by the $n$ columns of $(D\omega)$ indexed by $\vec{m}$.  The Jacobian is surjective (and thus $\omega$ a submersion) for all those $E$ for which $g(E) \neq 0$.  By analyticity of $g$, either $\{ E : g(E) = 0\}$ is (Lebesgue) measure zero in $\cK^n$ or $g$ is constantly zero.  If $g$ is constantly zero, then the set of critical points $X$ is the whole domain $\cK^n$.  Thus, if $g$ were constantly zero, then the image of $\omega$ must be of zero measure.  But this is certainly false; the bands of $n$-gap compacts can have arbitrary harmonic measures \cite{EY12, TOT09}.  

Thus, $\{ E : g(E) = 0\}$ has (Lebesgue) measure zero, and $\omega$ is almost-everywhere a submersion.  It follows that the preimage of a measure zero set has measure zero \cite{PON87}.
\end{proof}

Let $q \in \bbZ^n$, $k \in \bbZ$, and denote $B_{q,k} = \{v \in \bbR^n : q \cdot v = k\} \cap T^n$.  Of course, $B_{q,k}$ is a translation of a closed proper subspace of $T^n$, and thus is nowhere dense.  Denote by $B = \bigcup_{q \in \bbZ^n}\bigcup_{k \in \bbZ} B_{q,k}$.  This set is a countable union of nowhere dense sets, i.e. a meager set.  We claim that 
\begin{lem}
\label{lem:nwdense}
$\omega^{-1}(B_{q,k})$ is nowhere dense and has measure zero.
\end{lem}

\begin{proof}
$B_{q,k}$ is a closed subset of $T^n$ of positive codimension, and thus of measure zero.  By continuity, $\omega^{-1}(B_{q,k})$ is likewise closed, and by Lemma \ref{pr:omcts} $\omega^{-1}(B_{q,k})$ has measure zero.  A closed, measure zero subset of Euclidean space is nowhere dense.
\end{proof}

We can now address the
\begin{proof}[Proof of Proposition \ref{pr:genratind}]
Consider the set $A \subset \cK^n$ of $n$-gap compacts whose frequency vectors are rationally dependent; that is, $A = \omega^{-1}(B)$.  That $A$ is meager follows from the definition and Lemma \ref{lem:nwdense}.  Because $B$ is a countable union of zero measure sets, $B$ likewise has measure zero.  By Proposition \ref{pr:omcts}, $\omega^{-1}(B)$ has measure zero in $\cK^n$.  Finally, because $\omega$ is fibre-wise constant, for each $\alpha > 0$,
\begin{align*}
\omega^{-1}(B) = (\omega^{-1}(B) \cap \cK^n_\alpha) \times \bbR_+,
\end{align*}
and it follows that $\omega^{-1}(B) \cap \cK^n_\alpha$ is meager and has measure zero under the pushforward measure induced by $h_\alpha$.
\end{proof}

\section{The Toda Hierarchy}

Recall the definition of the Toda hierarchy from Section 2.  Namely, for a polynomial $P$ of degree $n+1 \geq 1$, recall that the $n^{th}$ Toda flow (for $P$) is the integral curve of Jacobi operators satisfying the Lax pair \eqref{eq:laxpair}
\begin{align*}
\partial_t J &= [P(J)^+ - P(J)^-, J].
\end{align*} 
When there exists a monic polynomial $P$ and Jacobi operator $J$ such that $P(J)^+ - P(J)^-$ and $J$ commute,
\begin{align*}
[P(J)^+ - P(J)^-, J] = 0,
\end{align*}
we say that $J$ is stationary for $P$.

The Toda flow exists uniquely for any bounded Jacobi initial condition $J_0$ \cite[Theorem 12.6]{TES00}.  Denote this flow by $\cT_P(t)$, that is, the Jacobi matrix $\cT_P(t) J_0$ should solve the Lax pair \eqref{eq:laxpair} with $\cT_P(0) J_0 = J_0$.  An important fact about the Toda flow which we did not mention above is the following:
\begin{prop}[Theorem 12.5, \cite{TES00}]
\label{pr:todauniteq}
Fix a polynomial $P$.  Then $\cT_P(t)J_0$ is unitarily equivalent to $J_0$ for all $t \in \bbR$.
\end{prop}
Thus, we see that the flow $\cT_P$ preserves spectral properties.  In fact, this flow also preserves the reflectionless condition \cite[Corollary 1.3]{PR11}, and thus, for fixed polynomial $P$, descends as a continuous-time dynamical system to $\cJ(E)$.  Amazingly, under the assumption that $E$ is homogeneous, the homeomorphism from Theorem \ref{thm:sy}, which linearized the shift dynamics, simultaneously linearizes the Toda flow.
\begin{thm}[Theorem 1.3, \cite{VY02}]
\label{thm:vy}
Suppose $E$ is a homogeneous compact with $g$ gaps, $0 \leq g \leq \infty$, and let $P$ be a polynomial.  Then the dynamical system $(\cJ(E), \cS, \cT_P(t))$ is conjugate to $((\bbR/\bbZ)^g,T_\omega, T_{t\xi_P})$, where $T_{t\xi_P}(\alpha) = \alpha + t\xi_P$, $\xi_P \in \bbR^g$.
\end{thm}
We remark that this result was well-known in the finite-gap regime \cite{GHMT08}, but we state this more general result for completeness.

Fix a compact $E \in \cK^n$ and consider the polynomial $Q_E(z)$ defined as in \eqref{eq:qe} above:
\begin{align*}
Q_E(z) = \prod_{j = 0}^n (z-E_j^+)(z-E_j^-).
\end{align*} 
Apply the square root with the typical branching and expand $Q_E^{1/2}$ for $|z| > \max\{ |E_0^+|, |E_0^-|\}$ as
\begin{align}
\label{eq:cjdef}
Q_E^{1/2}(z) &= -z^{n+1}\sum_{j=0}^\infty c_j(E)z^{-j}.
\end{align}

We intend to address the proof of property \textit{(3)} from Proposition \ref{pr:rem1}, given the following statement from \cite{TES00}:

\begin{prop}[Corollary 12.10, \cite{TES00}]
\label{pr:tes1}
Let $J \in \cJ(E)$.  Then $J$ is stationary for a Toda polynomial
\begin{align*}
P(z) = z^{m+1} - c_1z^m - \cdots - c_m z
\end{align*} 
of degree $m+1$ if and only if there exists a polynomial $Q_{m-n}$ of degree $m-n$ such that $c_j = c_j (\tilde{E})$, where $Q_{\tilde{E}}$ is given by
\begin{align*}
Q_{\tilde{E}}(z) = Q_{m-n}^2(z)Q_E(z).
\end{align*}
\end{prop}

\begin{prop}[Proposition \ref{pr:rem1}, \textit{(3)}]
\label{pr:tes1res}
If $J$ is stationary for $P$, then the spectrum of $J$ has at most $\deg(P) - 1$ (open) spectral gaps.  Furthermore, for each $E \in \cK^n$, there exists a polynomial $P$ of degree $n+1$ so that every $J \in \cJ(E)$ is stationary for $P$.
\end{prop}
\begin{proof}
Let $E = \sigma(J)$, and suppose $J$ is stationary for $P$ of degree $m+1$.  Since $J$ is stationary, $J \in \cJ(E)$ with $E$ finite-gap by Proposition \ref{pr:rem1} \textit{(2)}.  Suppose $E$ has $n$ gaps.  Then, by Proposition \ref{pr:tes1}, we have
\begin{align*}
P(z) = z^{m+1} - c_1z^m - \cdots - c_m z,
\end{align*}  
where $c_j = c_j(\tilde{E})$ are determined as in \eqref{eq:cjdef} for some polynomial $Q_{\tilde{E}}(z)$ of degree $2(m+1)$.  Furthermore, we have that $Q_E(z)$, a polynomial of degree $2(n+1)$, is a factor of $Q_{\tilde{E}}(z)$.  But for this to be true, it must be the case that $2(m+1) \geq 2(n+1)$; that is, that $m = \deg(P)-1 \geq n$.

Now, suppose $E \in \cK^n$, and define a sequence of coefficients $c_j = c_j(E)$ as in \eqref{eq:cjdef}.  Then taking $Q_{m-n}(z) = 1$ in Proposition \ref{pr:tes1} shows that $J$ is stationary for $P$.
\end{proof}

While we have introduced all of these results, we feel it necessary, while not particularly relevant to the matter at hand, to note a simple consequence that we have not observed in the literature:
\begin{thm}
Let $E \in \cK^n$ for some $n \geq 0$, and suppose $J, J' \in \cJ(E)$.  Then $J$ and $J'$ are unitary equivalent.
\end{thm}
\begin{proof}
By Theorem \ref{thm:vy}, Toda flows for various Toda polynomials are simultaneously linearized by our conjugacy.  Thus, by Proposition \ref{pr:todauniteq}, it suffices to find $n$ linearly independent Toda flows on $\cJ(E)$.  We claim that the Toda flows induced by $P_j(z) := z^j$, $1 \leq j \leq n$ have this property.

Suppose otherwise; then there exists a linear combination $P = \sum_{j=1}^n c_jP_j$ such that the vector $\xi_P$ (from Theorem \ref{thm:vy}) vanishes.  But if $\xi_P = 0$, then for any $J \in \cJ(E)$, $J$ is stationary for $P$.  But $\deg(P) \leq n$ and $E \in \cK^n$, contradicting Proposition \ref{pr:tes1res}.
\end{proof}


\def\cprime{$'$}

\end{document}